\documentclass[12pt, a4paper]{amsart}

\makeatletter
\@namedef{subjclassname@2020}{%
  \textup{2020} Mathematics Subject Classification}
\makeatother

\usepackage{amsthm, amsmath, amsfonts, amssymb, mathtools}
\usepackage{graphicx}
\usepackage{float}
\usepackage{enumerate}
\usepackage{xspace}
\usepackage{booktabs}
\usepackage[a4paper, inner=1.5in, outer=1in]{geometry}
\usepackage[all]{nowidow}

\usepackage{color}

\newcommand{\df}[1]{\ensuremath{\mathrm{df}({#1})}\xspace}
\newcommand{\core}{{\mathrm{core}}}
\newcommand{\bi}{\operatorname{bi}\xspace}
\newcommand{\odd}{\operatorname{odd}\xspace}

\newtheorem{theorem}{Theorem}[section]
\newtheorem{lemma}[theorem]{Lemma}
\newtheorem{proposition}[theorem]{Proposition}

\newtheorem{remark}[theorem]{Remark}
\newtheorem{example}[theorem]{Example}

\title%[Berge's conjecture for graphs with small defect]
[Berge's conjecture for cubic graphs\ldots]{Berge's conjecture for cubic graphs\\
with small colouring defect}%\tnoteref{ded}}
\author[J. Karab\'a\v{s}]{J\'an Karab\'a\v{s}}
\author[E. M\'a\v{c}ajov\'a]{Edita M\'a\v cajov\'a}
\author[R. Nedela]{Roman Nedela}
\author[M. \v{S}koviera]{Martin \v Skoviera}

\email[J. Karab\'a\v{s}]{jan.karabas@stuba.sk}
\email[E. M\'a\v{c}ajov\'a]{macajova@dcs.fmph.uniba.sk}
\email[R. Nedela]{nedela@savbb.sk}
\email[M. \v{S}koviera]{skoviera@dcs.fmph.uniba.sk}

\address[J. Karab\'a\v{s}]{
Faculty of Electrical Engineering and Information Technology,
Slovak University of Technology, Bratislava, Slovakia
}
\address[E. M\'a\v{c}ajov\'a]{
Faculty of Mathematics, Physics and Informatics, Comenius University, Bratislava, Slovakia
}
\address[R. Nedela]{
Faculty of Applied Sciences, University of West Bohemia, Pilsen, Czech Republic}

\address[J. Karab\'a\v{s}, R. Nedela]{
Mathematical Institute of Slovak Academy of Sciences,
Bansk\'a Bystrica, Slovakia
}
\address[M. \v{S}koviera]{
Faculty of Mathematics, Physics and Informatics, Comenius University, Bratislava, Slovakia
}

\begin{document}

\begin{abstract}
A long-standing conjecture of Berge suggests that every
bridgeless cubic graph can be expressed as a union of at most
five perfect matchings. This conjecture trivially holds for
$3$-edge-colourable cubic graphs, but remains widely open for
graphs that are not $3$-edge-colourable. The aim of this paper
is to verify the validity of Berge's conjecture for cubic
graphs that are in a certain sense close to $3$-edge-colourable
graphs. We measure the closeness by looking at the colouring
defect, which is defined as the minimum number of edges left
uncovered by any collection of three perfect matchings. While
$3$-edge-colourable graphs have defect $0$, every bridgeless
cubic graph with no $3$-edge-colouring has defect at least $3$.
In 2015, Steffen proved that the Berge conjecture holds for
cyclically $4$-edge-connected cubic graphs with colouring
defect $3$ or $4$. Our aim is to improve Steffen's result in
two ways. We show that all bridgeless cubic graphs with
defect $3$ satisfy Berge's conjecture irrespectively of their
cyclic connectivity. If, additionally, the graph in
question is cyclically $4$-edge-connected, then four perfect
matchings suffice, unless the graph is the Petersen graph. The
result is best possible as there exists an infinite family of
cubic graphs with cyclic connectivity $3$ which have defect $3$
but cannot be covered with four perfect matchings.
\end{abstract}



\dedicatory{Dedicated to professor J\'an Plesn\'\i{}k, our teacher and colleague.}

\maketitle

\section{Introduction}
\noindent{}In 1970's, Claude Berge made a conjecture that every
bridgeless cubic graph $G$ can have its edges covered by at
most five perfect matchings (see \cite{Mazz}). The
corresponding set of matchings is called a \textit{Berge cover}
of $G$.

Berge's conjecture relies on the fact that every preassigned
edge of $G$ belongs to a perfect matching, and hence there is a
set of perfect matchings that cover all the edges of $G$. The
smallest number of perfect matchings needed for this purpose is
the \textit{perfect matching index} of $G$, denoted by
$\pi(G)$. Clearly, $\pi(G)=3$ if and only if $G$ is
$3$-edge-colourable, so if $G$ has chromatic index~$4$, then
the value of $\pi(G)$ is believed to be either $4$ or $5$. In
this context it may be worth mentioning that $\pi(Pg)=5$, where
$Pg$ denotes the Petersen graph, and that cubic  graphs with
$\pi=5$ are very rare \cite{GGHM, MS-pmi}.

Berge's conjecture is closely related to a stronger and
arguably more famous conjecture of Fulkerson \cite{F},
attributed also to Berge and therefore often referred to as the
Berge-Fulkerson conjecture \cite{Seymour-multi}. The latter
conjecture suggests that every bridgeless cubic graph contains
a collection of six perfect matchings such that each edge
belongs to precisely two of them. Fulkerson's conjecture
clearly implies Berge's. Somewhat surprisingly, the converse
holds as well, which follows from an ingenious construction due
to Mazzuoccolo \cite{Mazz}.

Very little is known about the validity of either of these
conjectures. Besides the $3$-edge-colourable cubic graphs, the
conjectures are known to hold only for a few classes of graphs,
mostly possessing a very specific structure (see
\cite{Chen-Fan, FV-pmi, HaoCQZ-2009, HaoCQZ-2018,
LiuHaoCQZ-2021, ManSh,SW}). On the other hand, it has been proved
that if Fulkerson's conjecture is false, then the smallest
counterexample would be cyclically $5$-edge-connected (M\'a\v
cajov\'a and Mazzuoccolo~\cite{MM}).

In this paper we investigate Berge's conjecture under more
general assumptions, not relying on any specific structure of
graphs: our aim is to verify the conjecture for
all cubic graphs that are, in a certain sense, close to
$3$-edge-colourable graphs. In our case, the proximity to
$3$-edge-colourable graphs will be measured by the value of
their colouring defect. The \textit{colouring defect} of a
cubic graph, or just \textit{defect} for short, is the smallest
number of edges that are left uncovered by any set of three
perfect matchings. This concept was introduced and thoroughly
studied by Steffen \cite{S2} in 2015 who used the notation
$\mu_3(G)$ but did not coin any term for it. Since a cubic
graph has defect $0$ if and only if it is $3$-edge-colourable,
colouring defect can serve as a measure of
uncolourability of cubic graphs, along with resistance,
oddness, and other similar invariants recently studied by a
number of authors, see for example \cite{Allie, AlMaS,
FMS-survey}.

In \cite[Corollary~2.5]{S2} Steffen proved that the defect of
every bridgeless cubic graph that cannot be $3$-edge-coloured
is at least $3$, and can be arbitrarily large. He also proved
that every cyclically $4$-edge-connected cubic graph with
defect $3$ or $4$ satisfies Berge's conjecture
\cite[Theorem~2.14]{S2}.

We strengthen the latter result in two ways.

\medskip

First, we show that every bridgeless
cubic graph with defect $3$ satisfies the Berge
conjecture irrespectively of its cyclic connectivity.

\begin{theorem}\label{thm:1}
Every bridgeless cubic graph with colouring defect $3$
admits a Berge cover.
\end{theorem}

Second, we prove that if a graph with colouring defect $3$ is
cyclically $4$-edge-connected, then four perfect matchings are
enough to cover all its edges, except for the Petersen graph.

\begin{theorem}\label{thm:2}
Let $G$ be a cyclically $4$-edge-connected cubic graph with
defect~$3$. Then $\pi(G)=4$, unless $G$ is the Petersen graph.
\end{theorem}

As we shall see later, Theorem~\ref{thm:2} is best
possible in the sense that there exist infinitely many
$3$-connected cubic graphs with colouring defect $3$ that
cannot be covered with four perfect matchings.

The existence of a cover with four perfect matchings is known
to have a number of important consequences. For example, such a
graph satisfies the Fan-Raspaud conjecture \cite{FR}, admits a
$5$-cycle double cover and has a cycle cover of length
$4/3\cdot m$, where $m$ is the number of edges
\cite[Theorem~3.1]{S2}.

\medskip

Our proofs use a wide range of methods. Both main results rely
on Theorem~\ref{thm:6cuts} which describes the structure of a
subgraph resulting from the removal of a $6$-edge-cut from a
bridgeless cubic graph. Moreover, the proof of
Theorem~\ref{thm:2} establishes an interesting relationship
between the colouring defect and the bipartite index of a cubic
graph. We recall that the bipartite index of a graph is the
smallest number of edges whose removal yields a bipartite
graph. This concept was introduced by Thomassen in \cite{Th1}
and was applied in \cite{Th1, Th2} in a different context.

\medskip

Our paper is organised as follows. The next section summarises
definitions and results needed for understanding the rest of
the paper. Section~\ref{sec:arrays} provides a brief account of
the theory surrounding the notion of colouring defect. Specific
tools needed for the proofs of Theorems~\ref{thm:1}
and~\ref{thm:2} are established in Section~\ref{sec:tools}. The
two theorems are proved in Sections~\ref{sec:thm1}
and~\ref{sec:thm2}, respectively. The final section illustrates
that the condition on cyclic connectivity in
Theorem~\ref{thm:2} cannot be removed.

\section{Preliminaries}
\noindent{}\textbf{2.1. Graphs.} Graphs studied in this paper
are finite and mostly cubic (that is, 3-valent). Multiple edges
and loops are permitted. The \emph{order} of a graph $G$,
denoted by $|G|$, is the number of its vertices. A
\emph{circuit} in $G$ is a connected $2$-regular subgraph of
$G$. A $k$-\emph{cycle} is a circuit of length~$k$. The
\emph{girth} of $G$ is the length of a shortest circuit in $G$.

An \emph{edge cut} is a set $R$ of edges of a graph whose
deletion yields a disconnected graph. A common type of an edge
cut arises by taking a subset of vertices or an induced
subgraph $H$ of $G$ and letting $R$ to be the set $\delta_G(H)$
of all edges with exactly one end in $H$. We omit the subscript
$G$ whenever $G$ is clear from the context.

A connected graph $G$ is said to be \emph{cyclically
$k$-edge-connected} for some integer $k\ge 1$ if the removal of
fewer than $k$ edges cannot leave a subgraph with at least two
components containing circuits. The \emph{cyclic connectivity}
of $G$ is the largest integer $k$ not exceeding the cycle rank
of $G$ such that $G$ is cyclically $k$-edge-connected
(recall that cycle rank is also known as Betti number or
cyclomatic number, see \cite[Chapter~1.9]{D}). An edge cut $R$
in $G$ that separates two circuits from each other is
\emph{cycle-separating}. It is not difficult to see that the
set $\delta_G(C)$ leaving a shortest circuit $C$ of a cubic
graph $G$ is cycle-separating unless $G$ is the complete
bipartite graph $K_{3,3}$, the complete graph $K_4$, or the
\emph{$3$-dipole}, the graph which consists of two vertices and
three parallel edges joining them. Observe that an edge cut
formed by a set of independent edges is always
cycle-separating. Conversely, a cycle-separating edge cut of
minimum size is independent.

\medskip

\noindent{}\textbf{2.2. Edge colourings and flows.} An
\emph{edge colouring} of a graph $G$ is a mapping from the edge
set of $G$ to a set of colours. A colouring is \emph{proper} if
any two edge-ends incident with the same vertex receive
distinct colours. A \emph{$k$-edge-colouring} is a proper edge
colouring where the set of colours has $k$ elements. Unless
specified otherwise, our colouring will be assumed to be proper
and graphs to be \emph{subcubic}, that is, with vertices of
valency $1$, $2$, or~$3$.

There is a  standard method of transforming a
$3$-edge-colouring to another $3$-edge-colouring: it uses
so-called Kempe switches: Let $G$ be a subcubic graph endowed
with a proper $3$-edge-colouring $\sigma$. Take two distinct
colours $i$ and $j$ from $\{1,2,3\}$. An \emph{$(i,j)$-Kempe
chain} in $G$ (with respect to $\sigma$) is a non-extendable
walk $L$ that alternates edges coloured $i$ with those coloured
$j$. It is easy to see that $L$ is either a bicoloured circuit or
path starting and ending at the vertex of valency smaller than
$3$. The \emph{Kempe switch} along a Kempe chain produces a new
$3$-edge-colouring of $G$ by interchanging the colours~on~$L$.

It is often useful to regard $3$-edge-colourings of cubic
graphs as nowhere-zero flows. To be more precise, one can
identify each colour from the set $\{1,2,3\}$ with its binary
representation; thus $1=(0,1)$, $2=(1,0)$, and $3=(1,1)$.
Having done this, the condition that the three colours meeting
at every vertex $v$ are all distinct becomes equivalent to
requiring the sum of the colours at $v$ to be $0=(0,0)$ in
$\mathbb{Z}_2\times\mathbb{Z}_2$. The latter is nothing but
the Kirchhoff law for nowhere-zero
$\mathbb{Z}_2\times\mathbb{Z}_2$-flows. Recall that an
\textit{$A$-flow} on a graph $G$ is a pair $(D,\phi)$ where
$\phi$ is an assignment of elements of an abelian group $A$ to
the edges of $G$, and $D$ is an assignment of one of two
directions to each edge in such a way that, for every vertex
$v$ in $G$, the sum of values flowing into $v$ equals the sum
of values flowing out of~$v$ (\emph{Kirchhoff's law}). A
\emph{nowhere-zero} $A$-flow is one which does not assign $0\in
A$ to any edge of $G$. If each element $x\in A$ satisfies
$x=-x$, then $D$ can be omitted from the definition. It is well
known that the latter is satisfied if and only if $A\cong
\mathbb{Z}_2^n$ for some $n\ge 1$.

The following well-known statement is a direct consequence of
Kirchhoff's law.

\begin{lemma}{\rm (Parity Lemma)}\label{lem:par}
Let $G$ be a graph with maximum degree $3$ endowed with a
proper $3$-edge-colouring~$\xi$. If $H$ is a subgraph of $G$
such that every vertex of $H$ is trivalent in $G$, then
$$\sum_{e\in\delta_G(H)}\xi(e)=0.$$
Equivalently, the number of edges in $\delta_G(H)$ carrying any
fixed colour has the same parity as the size of the cut.
\end{lemma}

A cubic graph $G$ is said to be \emph{colourable} if it admits
a $3$-edge-colouring. A $2$-connected cubic graph that admits
no $3$-edge-colouring is called a \emph{snark}. Our definition
agrees with that of Cameron et al. \cite{CCW}, Nedela and \v
Skoviera \cite{NS-decred}, Steffen \cite{S1}, and others, and
leaves the concept of a snark as wide as possible. A more
restrictive definition requires a snark to be to be cyclically
$4$-edge-connected, with girth at least~$5$, see for
example~\cite{FMS-survey}. We call such snarks
\emph{nontrivial}.

\medskip

\noindent{}\textbf{2.3. Perfect matchings.} The classical
theorems of Tutte \cite{Tutte} and Plesn\'ik \cite{Plesnik},
stated below,
%and Plesn\'ik \cite{Plesnik},
will be repeatedly used throughout the paper, albeit in a
slightly modified form permitting parallel edges and loops.
These extensions can be proved easily by using the standard
versions of the corresponding theorems.

Let $\odd(G)$ denote the number of \emph{odd components} of
$G$, that is, the components with an odd number of vertices.

\begin{theorem} \label{thm:Tutte}
{\rm (Tutte, 1947)} A graph $G$, possibly containing parallel
edges and loops, has a perfect matching if and only if
$$\odd(G-S)\le |S| \quad \text{for all }S\subseteq V(G).$$
\end{theorem}

\begin{theorem} \label{thm:Plesnik}
{\rm (Plesn\'\i k, 1972)} Let $G$ be an $(r-1)$-edge-connected
$r$-regular graph with $r\ge 1$, and let $A$ be an arbitrary
set of $r-1$ edges in $G$. If $G$ has even order, then $G - A$
has a perfect matching.
\end{theorem}

\section{Colouring defect of a cubic graph}
\label{sec:arrays}

\noindent{}In this section we discuss a number of structures
related to the concept of colouring defect, which will be
used in the proofs of Theorems~\ref{thm:1} and~\ref{thm:2}.
More details on this matter can be found in
 %\cite{KMNSred} and
\cite{S2}.

For a bridgeless cubic graph $G$ we define a \emph{$k$-array of
perfect matchings}, or briefly a \emph{$k$-array} of $G$, as an
arbitrary collection $\mathcal{M}=\{M_1, M_2, \ldots, M_k\}$ of
$k$ perfect matchings of $G$, not necessarily pairwise
distinct. The concept of a $k$-array unifies a number of
notions, such as Berge covers, Fulkerson covers, Fan-Raspaud
triples, and others. Our main concern here are $3$-arrays,
which we regard as approximations of $3$-edge-colourings.

Let $\mathcal{M}=\{M_1, M_2, M_3\}$ be a $3$-array of perfect
matchings of a cubic graph $G$. An edge of $G$ that belongs to
at least one of the perfect matchings of $\mathcal{M}$ will be
considered to be \emph{covered}. An edge will be called
\emph{uncovered}, \emph{simply covered}, \emph{doubly covered},
or \emph{triply covered} if it belongs, respectively, to zero,
one, two, or three distinct members of~$\mathcal{M}$.

Given a cubic graph $G$, it is natural task to maximise the
number of covered edges, or equivalently, to minimise the
number of uncovered ones. A $3$-array that leaves the minimum
number of uncovered edges will be called \emph{optimal}. The
number of edges left uncovered by an optimal $3$-array is the
\emph{colouring defect} of $G$, or the \emph{defect} for short,
denoted by $\df{G}$.

Let $\mathcal{M}=\{M_1, M_2, M_3\}$ be a $3$-array of perfect
matchings of a cubic graph $G$. One way to describe
$\mathcal{M}$ is based on regarding the indices $1$, $2$, and
$3$ as colours. It will therefore be convenient to
identify the indices with the nonzero elements of the group
$\mathbb{Z}_2\times\mathbb{Z}_2$. Since the same edge may
belong to more than one member of $\mathcal{M}$, an edge of $G$
may receive more than one colour. To each edge $e$ of $G$ we
can therefore assign the list $\phi(e)$ of all colours  in
lexicographic order it receives from $\mathcal{M}$. We let
$w(e)$ denote the number of colours in the list $\phi(e)$ and
call it the \emph{weight} of $e$ (with respect to
$\mathcal{M}$). In this way $\mathcal{M}$ gives rise to a
colouring
$$\phi\colon E(G)\to\{\emptyset, 1, 2, 3, 12, 13, 23, 123\}$$
where $\emptyset$ denotes the empty list. Obviously, a mapping
$\sigma$ assigning subsets of $\{1,2,3\}$ to the edges of $G$
determines a $3$-array if and only if, for each vertex $v$ of
$G$, each element of  $\{1,2,3\}$ occurs precisely once in the
subsets of $\{1,2,3\}$ assigned by $\sigma$ to the edges
incident with $v$. In general, $\phi$ need not be a proper
edge-colouring. However, since $M_1$, $M_2$, and $M_3$ are
perfect matchings, the only possibility when two edges equally
coloured under $\phi$ meet at a vertex is that both of them
receive colour $\emptyset$. In this case the third edge
incident with the vertex is coloured $123$.

A different but equivalent way of representing a $3$-array uses
a mapping
\[
\chi\colon E(G)\to \mathbb{Z}_2^3,
\quad e\mapsto \chi(e)=(x_1, x_2, x_3)
\]
defined by setting $x_i = 0$ if and only if $e\in M_i$, where
$i\in\{1,2,3\}$. Thus $x_i=1$ if and only if $e$ belongs
to the $2$-factor complementary to $M_i$, which implies that
$\chi$ is a $\mathbb{Z}_2^3$-flow. We call $\chi$ \emph{the
characteristic flow} for $\mathcal{M}$. Again, $\chi$ is a
nowhere-zero $\mathbb{Z}_2^3$-flow if and only if $G$ contains
no triply covered edge.

In the context of $3$-arrays the characteristic flow was
introduced in~\cite[p.~166]{JSM}. Observe that the
characteristic flow $\chi$ of a $3$-array and the colouring
$\phi$ determine each other. In particular, the condition on
$\phi$ requiring all three indices from $\{1,2,3\}$ to occur
precisely once in a colour around any vertex is equivalent to
Kirchhoff's law.

The following result characterises $3$-arrays with no triply
covered edge.

\begin{proposition}\label{prop:notriply}
Let $\mathcal{M}$ be a $3$-array of perfect matchings of a
cubic graph $G$. The following three statements are equivalent.
\begin{enumerate}[{\rm (i)}]
\item $G$ has no triply covered edge with respect to
    $\mathcal{M}$.
\item The associated colouring $\phi\colon
    E(G)\to\{\emptyset, 1,2, 3, 12, 13, 23, 123\}$ is
    proper.
\item The characteristic flow $\chi$ for $\mathcal{M}$,
    with values in $\mathbb{Z}_2^3$, is nowhere-zero.
\end{enumerate}
\end{proposition}

%\medskip

The next important structure associated with a $3$-array is its
core. The \emph{core} of a $3$-array $\mathcal{M}=\{M_1, M_2,
M_3\}$ of $G$ is the subgraph of $G$ induced by all the edges
of $G$ that are not simply covered; we denote it by
$\core(\mathcal{M})$. The core is called \emph{optimal}
whenever $\mathcal{M}$ is optimal.

It is worth mentioning that if $G$ is $3$-edge-colourable and
$\mathcal{M}$ consists of three pairwise disjoint perfect
matchings, then $\core(\mathcal{M})$ is empty. If $G$ is not
$3$-edge-colourable, then every core must be nonempty.
Figure~\ref{fig:petersen_core} shows the Petersen graph endowed
with a $3$-array whose core is its  ``outer'' $6$-cycle. The
hexagon is in fact an optimal core.

\begin{figure}[h!]
 \centering
 \includegraphics[scale=1.4]{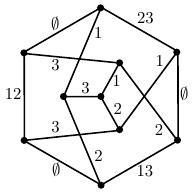}
 \caption{An optimal $3$-array of the Petersen graph}
\label{fig:petersen_core}
\end{figure}

The following proposition, due to Steffen \cite[Lemma~2.2]{S2},
describes the structure of optimal cores in the general case.

\begin{proposition}\label{prop:core}
Let $\mathcal{M}=\{M_1, M_2, M_3\}$ be an optimal $3$-array of
perfect matchings of a snark $G$. Then  every component of
$\core(\mathcal{M})$ is either an even circuit of length at
least $6$ or a subdivision of a cubic graph. Moreover, the
union of doubly and triply covered edges forms a perfect
matching of $\core(\mathcal{M})$.
\end{proposition}

\medskip

The next theorem characterises snarks with minimal possible
colouring defect. The lower bound for the defect of a snark --
the value $3$ -- is due to Steffen~\cite[Corollary~2.5]{S2}.

\begin{theorem}\label{thm:main}
Every snark $G$ has $\df{G}\ge 3$. Furthermore, the following
three statements are equivalent.
\begin{enumerate}[{\rm(i)}]
\item $\df{G}=3$.
\item The core of any optimal $3$-array of $G$ is a
    $6$-cycle.
\item $G$ contains an induced $6$-cycle $C$ such that the
    subgraph $G-E(C)$ admits a proper $3$-edge-colouring
    under which the six edges of $\delta(C)$ receive
    colours $1,1,2,2,3,3$
    %or $1,2,2,3,3,1$
    with respect to the cyclic order induced by
    an orientation of $C$.
\end{enumerate}
\end{theorem}

\begin{figure}[h!]\centering
\includegraphics[scale=1.8]{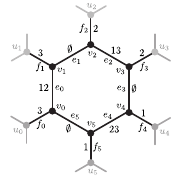}%\label{fig:oddness1}
\caption{The hexagonal core and its vicinity.}
% The hexagonal core for of a snark with defect $3$ and its vicinity.
\label{fig:core3}
\end{figure}

If $G$ is an arbitrary snark with $\df{G} = 3$, then, by
Theorem~\ref{thm:main}~(iii), $G$ contains an induced $6$-cycle
$C=(e_0e_1e_2e_3e_4e_5)$ such that $G-E(C)$ is
$3$-edge-colourable. We say that that $C$ is a \emph{hexagonal
core} of $G$.

%It is easy to see that any $3$-edge-colouring of $G-E(C)$
%induced by a $3$-array of $G$ extends to an optimal $3$-array
%of $G$ whose core is $C$; we say

Let $f_i$ denote the edge of $\delta(C)$ which is incident with
$e_{i-1}$ and $e_i$, where $i\in\{0,1,\ldots,5\}$ and the
indices are reduced modulo $6$; see Figure~\ref{fig:core3}.
Since $G-E(C)$ is $3$-edge-colourable but $G$ is not, it is not
difficult to see that for every $3$-edge-colouring
%$\sigma$
of $G-E(C)$ the cyclic order of colours around $C$ is
$(1,1,2,2,3,3)$ up to permutation of colours. Moreover, we can
assume that the values of the associated colouring $\phi$ of
$G$ induced by $\mathcal{M}$ in the vicinity of $C$ are those
as shown in Figure~\ref{fig:core3}, or can be obtained from
them by the rotation one step clockwise. Note that the two
possibilities only depend on the position of the uncovered
edges. In any case, there are no triply covered edges, and so
$\phi$ is a proper edge colouring due to
Proposition~\ref{prop:core}.

\section{Tools}\label{sec:tools}
\noindent{}In this section we establish tools for proving
Theorems~\ref{thm:1} and~\ref{thm:2}. Its main results are
stated as Theorems~\ref{thm:6cuts} and~\ref{thm:abg_class1},
both having somewhat ``bipartite flavour''. The first theorem
suggests that if a subgraph $H$ of a bridgeless cubic graph is
separated from the rest by a $6$-edge-cut and has no perfect
matching, then the structure of $H$ is -- essentially -- that
of a bipartite cubic graph. Moreover, only one of the two parts
is incident with the cut.

\begin{theorem}\label{thm:6cuts}
Let $G$ be a bridgeless cubic graph and let $H\subseteq G$ be a
subgraph with $|\delta_G(H)|=6$. Then $H$ has a perfect
matching, or else $H$ contains an independent set $S$ of
vertices such that
\begin{enumerate}[{\rm (i)}]
\item every vertex of $S$ is trivalent in $H$,
\item $\odd(H-S) =|S|+ 2$, and
\item $|\delta_G(L)| = 3$ for each component $L$ of $H-S$.
\end{enumerate}
\end{theorem}

\begin{proof}
Set $K=G-V(H)$. Assume that $H$ has no perfect matching. Tutte's
Theorem tells us that there exists a set $S\subseteq V(H)$ such
that $\odd(H-S)>|S|$. As $H$ has an even number of vertices,
the numbers $|S|$ and $\odd(H-S)$ have the same parity, so
\begin{equation}\label{eq:loc1}
\odd(H-S)-|S|\ge 2.
\end{equation}
Set $a=|\delta_G(S)\cap\delta_G(K)|$. Clearly,
$a\in\{0,1,\ldots,6\}$. Since $|\delta_H(S)|+a=|\delta_G(S)|\le
3|S|$, we get
\begin{equation}\label{eq:dS-le}
|\delta_H(S)|\le 3|S|-a.
\end{equation}
To bound $|\delta_H(S)|$ from below, we first realise that each
odd component of $H-S$ is incident with at least three edges of
$\delta_G(H-S)$ because $G$ is bridgeless. Moreover, there are
$6-a$ edges joining $H-S$ to $K$ (see Figure~\ref{fig:tutte}).
Therefore
\begin{equation}\label{eq:dS-ge}
|\delta_H(S)|=|\delta_H(H-S)|=|\delta_G(H-S)|-(6-a)\ge 3\cdot\odd(H-S)-(6-a).
\end{equation}
If we combine \eqref{eq:dS-le} with \eqref{eq:dS-ge}, we get
\begin{equation}\label{eq:loc2}
3\cdot\odd(H-S)-6+a \le |\delta_H(S)| \le 3|S|-a.
\end{equation}
Since $3\cdot\odd(H-S)-3|S|\ge 6$ according to \eqref{eq:loc1}, we
can rewrite \eqref{eq:loc2} as
$$-2a\ge 3\cdot\odd(H-S)-3|S|-6\ge 0,$$
which implies that $a=0$.

Now we insert $a=0$ into
\eqref{eq:loc2} and get $3\cdot\odd(H-S)-3|S|\le 6$. Together
with \eqref{eq:loc1}, multiplied by $3$, this yields that
$$3\cdot\odd(H-S)-6\le |\delta_H(S)| \le 3|S|\le 3\cdot\odd(H-S)-6.$$
\begin{figure}[H]
\centering
\includegraphics[width=0.7\textwidth]{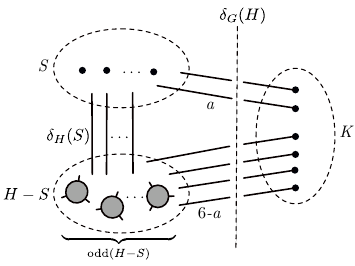}
\caption{The structure of $G$ as described in the proof of
Theorem~\ref{thm:6cuts}.}
\label{fig:tutte}
\end{figure}
\noindent
Hence,
\begin{equation}\label{eq:S+2}
|\delta_H(S)|=3|S|=3\cdot\odd(H-S)-6,
\end{equation}
which implies that $S$ is an independent set of $H$, with all
its vertices $3$-valent, and that each component $L$ of $H-S$
is odd with $|\delta_G(L)|=3$.  Thus we have proved Statements
(i) and (iii). Moreover, Equation~\eqref{eq:S+2} implies
(ii). The proof is complete.
\end{proof}

\begin{example}
{\rm We illustrate Theorem~\ref{thm:6cuts} in two simple
instances. First, if $G$ is the $3$-dimensional cube $Q_3$ and
$H=G-V(C)$, where $C$ is an induced $6$-cycle of $G$, then
Theorem~\ref{thm:6cuts} holds with $S=\emptyset$. If $G$ is the
Petersen graph $Pg$ and $H= G-V(C)$, where $C$ is again a
$6$-cycle, then $H$ is isomorphic to the complete bipartite
graph $K_{1,3}$ and $S$ is constituted by its central vertex.}
\end{example}

We proceed to our second tool, which describes cubic graphs
just one step away from being bipartite. Recall that if a cubic
graph is bipartite, then it is obviously bridgeless. Moreover,
if it is connected, then its bipartition is uniquely
determined. On the other hand, if a cubic graph is not
bipartite, then, clearly, at least two edges have to be removed
in order to produce a bipartite graph. Motivated by these two
facts we define a cubic graph $G$ to be \emph{almost bipartite}
if it is bridgeless, not bipartite, and contains two edges $e$
and $f$ such that $G-\{e,f\}$ is a bipartite graph. The edges
$e$ and $f$ are said to be \emph{surplus edges} of $G$. If a
cubic graph is almost bipartite, then  there exists a component
$K$ of $G$ such that the surplus edges connect vertices within
different partite sets of $K$.

Our aim is to show that every almost bipartite graph is
$3$-edge-colourable. We start with the following.

\begin{proposition}\label{prop:almost_bip}
Every almost bipartite cubic graph has a perfect matching that
contains both surplus edges.
\end{proposition}

\begin{proof}
Let $G$ be an almost bipartite cubic graph with surplus edges
$e$ and $f$. Since $e$ and $f$ belong to the same component of
$G$, we may assume that $G$ is connected. Let $\{A,B\}$ be the
bipartition of $G'=G-\{e,f\}$. As said before, the endvertices
of one of the surplus edges, say $e$, belong to $A$ and those
of $f$ then belong to $B$. In particular, this means that
$|A|=|B|=n$ for some positive integer $n$. By
Theorem~\ref{thm:Plesnik}, there exists a perfect matching $M$
containing the edge~$e$. There are $n-2$ edges of $M$ that
match $n-2$ vertices of $A$ to $n-2$ vertices of $B$. It
follows that $f\in M$.
\end{proof}

We would like to mention that our
Proposition~\ref{prop:almost_bip} has been inspired by
Lem\-ma~3~(iii) of \cite{RLER}. In principle, the mentioned
lemma could be used to prove Proposition~\ref{prop:almost_bip},
however, the proof would not be anything close to being
straightforward.

\medskip

We are now ready for the following theorem.

\begin{theorem}\label{thm:abg_class1}
Every almost bipartite cubic graph is $3$-edge-colourable.
\end{theorem}

\begin{proof}
Let $G$ be an almost bipartite graph with surplus edges $e$ and
$f$. By Proposition~\ref{prop:almost_bip}, $G$ has a perfect
matching $M$ which includes both $e$ and $f$. Now, $G-M$ is
bipartite $2$-regular spanning subgraph of $G$, so each
component of $G-M$ is an even circuit. It means that $G-M$ can
be decomposed into two disjoint perfect matchings $N_1$ and
$N_2$. Put together, $M$, $N_1$, and $N_2$ are colour classes
of a proper $3$-edge-colouring of $G$.
\end{proof}

The previous theorem links the problem of
$3$-edge-colourability of cubic graphs to an important concept
of a bipartite index, which has been introduced to measure how
far a graph is from being bipartite. Following \cite[Definition
2.4]{Th2}, we define the \emph{bipartite index} {$\bi(G)$ of a
graph $G$ to be the smallest number of edges that must be
deleted in order to make the graph bipartite. In other words,
the bipartite index of a graph is the number of edges outside a
maximum cut (by which we mean the edge cut $\delta_G(X)$
with  $|\delta_G(X)|$ maximum).

\begin{figure}[H]
\centering
\includegraphics[width=0.35\textwidth]{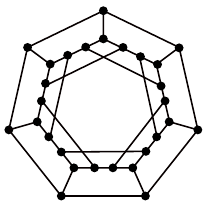}
\caption{The flower snark $J_7$}
\label{fig:j7}
\end{figure}

Our Theorem~\ref{thm:abg_class1} states that all bridgeless
cubic graphs with bipartite index at most two are $3$-edge
colourable. This is, in fact, best possible as there exist
infinitely many nontrivial snarks whose bipartite index equals
$3$: this is true, for example, for the Petersen graph or for
all Isaacs flower snarks $J_{2n+1}$; see \cite{Isaacs} for the
definition and Figure~\ref{fig:j7} for $J_7$.

Theorem~\ref{thm:abg_class1} thus suggests that, along with
oddness, resistance, flow resistance, and other similar
invariants extensively studied in \cite{FMS-survey}, bipartite
index can serve as another measure of uncolourability of cubic
graphs. It is easy to see, for example, that for a cubic graph
$\bi(G)\geq \omega(G)$, where $\omega(G)$ denotes the
\emph{oddness} of a cubic graph, that is, the smallest number
of odd circuits in a $2$-factor of $G$. Since there exist
snarks of arbitrarily large oddness \cite{LMMS}, there are
snarks of arbitrarily large bipartite index.

\section{Proof of Theorem~\ref{thm:1}}\label{sec:thm1}

\noindent{}The purpose of this section is to establish the
existence of a Berge cover for every bridgeless cubic graph of
defect $3$ (Theorem~\ref{thm:1}) and for every cyclically
$4$-edge-connected cubic graph of defect $4$. We begin with the
following auxiliary result.

\begin{lemma}\label{lem:M4}
Let $G$ be a cubic graph of defect $3$ and let
$\mathcal{M}=\{M_1,M_2,M_3\}$ be an optimal $3$-array of
perfect matchings for $G$. Then $G$ has a fourth perfect
matching $M_4$ which covers at least two of the three edges
left uncovered by $\mathcal{M}$.
\end{lemma}

\begin{proof}
Consider the $6$-cycle $C=(e_0e_1\ldots e_5)$ constituting the
core of $\mathcal{M}$. We adopt the notation introduced in
Figure~\ref{fig:core3}; in particular, $e_1$, $e_3$, and $e_5$
are the three uncovered edges of~$C$. We also assume that the
edge colouring $\phi$ associated with $\mathcal{M}$ takes
values as indicated in Figure~\ref{fig:core3}. By
Proposition~\ref{prop:notriply}, the colouring is proper.

An uncovered edge $e_i$ of $C$ will be considered \emph{bad} if
$G$ has a cycle-separating $3$-edge-cut $R_i$ containing the
edges $e_{i-1}$ and $e_{i+1}$. We claim that not all of the
edges $e_1$, $e_3$, and $e_5$ can be bad. Suppose to the
contrary that all of them are bad. For $i\in\{1,3,5\}$ let
$h_i$ denote the third edge of the cut $R_i$. Since the edges
$f_i$ and $f_{i+1}$ of $\delta_G(C)$ are both adjacent to
$e_i$, the set $R_i'=\{f_{i}, f_{i+1},h_i\}$ is also a
$3$-edge-cut. Moreover, all three edges of $R_i'$ are simply
covered. Let $L_i$ be the component of $G-R_i'$ that does not
contains $e_i$.

Recall that the restriction of $\phi$ to $G-E(C)$ is a proper
$3$-edge-colouring. We show that this colouring can be modified
and extended to a $3$-edge-colouring $\psi$ of the entire~$G$.
By Kirchhoff's law, $\phi(f_1)+\phi(f_2)+\phi(h_1)=0$. Since
$\phi(f_1)=3$, and $\phi(f_2)=2$, we have $\phi(h_1)=1$.
Similarly, $\phi(h_3)=3$ and $\phi(h_5)=2$. It follows
that the edges $h_1$, $h_2$, and $h_3$ are pairwise distinct
and constitute a $3$-edge-cut; see Figure~\ref{fig:constr}. To
produce the colouring $\psi$ we proceed as follows. First, for
each component $L_i$, where $i\in\{1,3,5\}$, we swap the
colours of the edges $f_i$ and $f_{i+1}$ by employing a Kempe
chain through $L_i$. The pairs of edges $(f_i, f_{i+1})$ from
$\delta_G(L_i)$ are now coloured $(2,3)$, $(1,2)$, and $(3,1)$,
respectively. Note that the colours on the hexagon $C$ have not
been affected. Now we can define $\psi$. If $\phi'$ denotes the
current assignment of colours, for each uncovered edge $e_i$,
where $i\in\{1,3,5\}$, we set
$\psi(e_i)=\phi'(f_i)+\phi'(f_{i+1})$. For each doubly covered
edge $e_i$, where $i\in\{0,2,4\}$, we set $\psi(e_i)=p+q$
whenever $\phi(e_i)=pq$. It is easy to check that $\psi$ is
indeed a proper $3$-edge-colouring of $G$. Since $G$ is a
snark, this cannot occur. Therefore, at least one uncovered
edge of $C$ must not be bad.

\begin{figure}[h]
\centering
\includegraphics[width=0.8\textwidth]{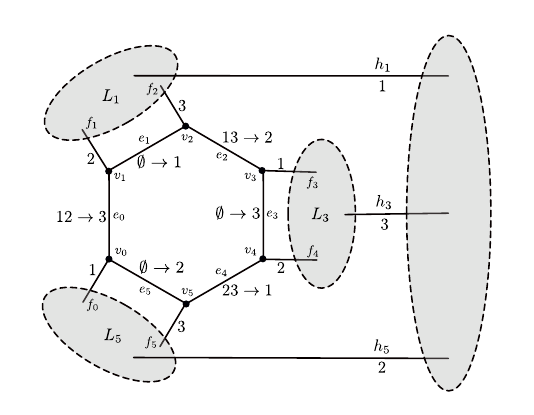}
\caption{Constructing a $3$-edge-colouring of $G$ provided
that $G$ has three bad edges; the situation after Kempe switches.}
\label{fig:constr}
\end{figure}

Without loss of generality we may assume that the edge $e_1$ is
not bad. Take the path $P=e_3e_4e_5\subseteq C$ and set
$H=G-V(P)$. Clearly, $|\delta_G(H)|=6$, so
Theorem~\ref{thm:6cuts} applies. We claim that $H$ admits a
perfect matching. Suppose not. Then, by
Theorem~\ref{thm:6cuts}, $H$ contains an independent set $S$ of
trivalent vertices such that each component $L$ of $H-S$ has
$|\delta_G(L)|=3$ and each edge of $\delta_G(H)$ joins a vertex
of $H-S$ to a vertex of $P$. Since the edge $e_1$ belongs to
$H-S$, the component containing $e_1$ is nontrivial. It follows
that $e_1$ is bad, and we have arrived at a contradiction. Thus
$H$ admits a perfect matching, and this matching is readily
extended to a perfect matching $M_4$ of $G$ that covers the
uncovered edges $e_3$ and $e_5$. The proof is complete.
\end{proof}

\begin{proof}[Proof of Theorem~\ref{thm:1}]
Take an optimal $3$-array $\{M_1,M_2,M_3\}$ for $G$. It leaves
three uncovered edges. According to Lemma~\ref{lem:M4}, there
is a perfect matching $M_4$ that covers at least two of them.
The remaining uncovered edge (if any) can be covered by a
perfect matching $M_5$ guaranteed by Theorem~\ref{thm:Plesnik}.
Together these perfect matchings constitute a Berge cover of
$G$.
\end{proof}

In the rest of this section we prove that every cyclically
$4$-edge-connected cubic graph with defect $4$ admits a Berge
cover. Our proof significantly differs from the one provided by
Steffen \cite[Theorem~2.14]{S2} in that it avoids the use of
Seymour's results on $p$-tuple multicolourings
\cite{Seymour-multi} by employing Theorem~\ref{thm:6cuts}
instead.

First we establish the following lemma.

\begin{lemma}\label{lem:cycliccore}
Let $G$ be a cyclically $4$-edge-connected cubic graph. If $G$
has an optimal $3$-array whose core is a circuit of length $d$,
then $\pi(G)\le 3+\lceil d/4\rceil$.
\end{lemma}

\begin{proof}
Let $\mathcal{M}$ be an optimal $3$-array of $G$ whose core is
a circuit $C$ of length $d$. By Proposition~\ref{prop:core},
$C$ alternates uncovered and doubly covered edges; in
particular, $d$ is even. Note that $d\ge 6$ by
Proposition~~\ref{prop:core}.

Let $P=efg\subseteq C$ be an arbitrary path of length $3$ whose
middle edge is doubly covered. We show that $G$ has a perfect
matching containing both uncovered edges of $P$. Set
$H=G-V(P)$. If $H$ has a perfect matching, say $M$, then
$M\cup\{e,g\}$ is a perfect matching of $G$ containing both $e$
and $g$, and we are done. Thus we may assume that $H$ has no
perfect matching. By Theorem~\ref{thm:6cuts}, there exists an
independent set $S$ of trivalent vertices of $H$ such
that $H-S$ satisfies Items (i)--(iii) of the theorem. In
particular, each component $L$ of $H-S$ is odd and has
$|\delta_G(L)|=3$. Since $G$ has no non-trivial $3$-cuts, each
component is just a vertex, implying that $H$ is a bipartite
graph. Since $P$ is connected to only one partition set of
$H$, it follows that $C\cap H$ is a path of even length. Hence,
$|E(C)|=|E(C\cap H)|+5$ is an odd number, which is impossible.
We have thus proved that there exists a perfect matching
containing any two uncovered edges of $C$ joined by a doubly
covered edge.

We are ready to finish the proof. It is easy to see that $C$
contains $\lfloor d/4\rfloor$ pairwise edge-disjoint paths of
length 3 that begin and end with an uncovered edge. At most one
uncovered edge remains outside these paths, which means that
the uncovered edges of $C$ can be covered by at most $\lceil
d/4\rceil$ paths of length 3 that begin and end with an
uncovered edge. In other words, there exists a set
$\mathcal{S}$ of $\lceil d/4 \rceil$ perfect matchings that
collectively contain all uncovered edges. Hence, $\pi(G)\le
|\mathcal{M}\cup\mathcal{S}|\le 3+\lceil d/4 \rceil$, as
claimed.
\end{proof}

\begin{theorem} Every cyclically $4$-edge-connected
cubic graph of defect $4$ admits a Berge cover.
\end{theorem}

\begin{proof}
Let $G$ be a cyclically $4$-edge-connected cubic graph of
defect $4$, let $\mathcal{M}=\{M_1,M_2,M_3\}$ be an optimal
$3$-array for $G$, and let $C$ be the core of $\mathcal{M}$. We
claim that $C$ does not contain a triply covered edge. Suppose
to the contrary that $e=uv$ is a triply covered edge in $C$.
Then there are four uncovered edges $uu_1$, $uu_2$, $vv_1$, and
$vv_2$ adjacent to $e$. Since $G$ is cyclically
$4$-edge-connected, the vertices $u$, $u_1$, $u_2$, $v$, $v_1$,
and $v_2$ are pairwise distinct. There are no other uncovered
edges in $G$ than the mentioned four, which leaves us, up to
isomorphism, with only two possibilities for~$C$:
\begin{itemize}
 \item[(1)] The core has doubly covered edges $u_1u_2$ and $v_1v_2$.
 \item[(2)] The core has doubly covered edges $u_1v_1$ and $u_2v_2$.
\end{itemize}

In the former case, $G$ contains the triangles $uu_1u_2$ and
$vv_1v_2$, which violates cyclic $4$-edge-connectivity. In the
latter case, we may clearly assume that the doubly covered edge
$u_1v_1$ belongs to $M_1\cap M_2$. Set
$M_2'=(M_2-\{uv,u_1v_1\})\cup\{uu_1,vv_1\}$. Then $M_2'$ is a
perfect matching of $G$ such that the $3$-array
$\mathcal{M}'=\{M_1,M_2',M_3\}$ leaves only two uncovered
edges, namely $uu_2$ and $vv_2$. Thus $\mathcal{M}$ is not
optimal, contrary to the assumption. All this shows that $C$
does not contain a triply covered edge. By
Proposition~\ref{prop:core}, each component of $C$ is an even
circuit of length at least $6$, which implies that $C$ is an
$8$-cycle. The result now follows from
Lemma~\ref{lem:cycliccore}.
\end{proof}

\section{Proof of Theorem~\ref{thm:2}}\label{sec:thm2}
\label{sec:c4c}

\noindent{}We prove that every cyclically $4$-edge-connected
cubic graph $G$ of defect $3$ has $\pi(G)=4$ with the only
exception of the Petersen graph.

\begin{proof}[Proof of Theorem~\ref{thm:2}]
Let $G$ be a cyclically $4$-edge-connected cubic graph with
$\df{G}=3$ and $\pi(G)>4$. Our aim is to show that $G$ is
isomorphic to the Petersen graph.

Let $C=(e_0e_1\ldots e_5)$ be the hexagonal core of an optimal
$3$-array $\mathcal{M}=\{M_1,M_2,\penalty0 M_3\}$ of $G$, and
let $H=G-V(C)$. We may assume that the colouring $\phi$
associated with $\mathcal{M}$ and the names of the vertices and
edges in the vicinity of $C$ are as in Figure~\ref{fig:core3}.
In particular, the endvertices of the edges of the cut
$\delta_G(C)=\{f_0,f_1,\ldots,f_5\}$ not lying on $C$ are
$u_0$, $u_1$, \ldots, $u_5$ and are listed in a cyclic order
around $C$. Set $U=\{u_0,\ldots,u_5\}$. Note that some of the
vertices $u_i$ and $u_j$, where $i\neq j$, may coincide.

\medskip

\noindent Claim 1. \emph{The subgraph $H$ is bipartite.}

\medskip\noindent
\emph{Proof of Claim 1.} If $H$ contains a perfect matching,
then the matching can be extended to a perfect matching $M_4$
of $G$ which covers the three uncovered edges of $C$. It
follows that $\pi(G)=4$, contrary to our assumption. Therefore
$H$ has no perfect matching. In this situation
Theorem~\ref{thm:6cuts} tells us that there exists a set
$S$ of trivalent vertices of $H$ such that each component
$L$ of $H-S$ has $|\delta_G(L)|=3$ and each edge of
$\delta_G(H)$ joins a vertex of $H-S$ to a vertex of $C$. As
$G$ is cyclically 4-edge-connected, each component of $H-S$ is
a single vertex. Consequently, $H$ is a bipartite graph with
bipartition $\{S, V(H)-S\}$ and edge set $\delta_G(S)$.  This
establishes Claim~1.

\medskip
Next, we explore the colouring properties of the subgraph
$H^+=G-E(C)\subseteq G$. Note that
$H^+=H+\{f_0,f_1,\ldots,f_5\}$ and $H^+\cup C=G$. Set
$C^+=C+\{f_0,f_1,\ldots,f_5\}$. In $C^+$, the endvertices
$u_0$, $u_1$, \dots, $u_5$ are assumed as pairwise distinct.

Before proceeding further we need several definitions. A
\emph{colour vector} is any sequence $\alpha=c_0c_1\ldots c_5$
of six colours $c_i\in\{1,2,3\}$ such that $c_0+\ldots+c_5=0$
in $\mathbb{Z}_2\times\mathbb{Z}_2$. It follows from
Parity Lemma that every $3$-edge-colouring $\sigma$ of $H^+$
thus \emph{induces} the colour vector
$\alpha_{\sigma}=\sigma(f_0)\sigma(f_1)\ldots\sigma(f_5)$. By
permuting the colours $1$, $2$, and $3$ in $\alpha$ we obtain a
new colour vector, nevertheless, the difference between the two
is insubstantial. Therefore each of the six permutations of
colours produces a sequence that encodes essentially the same
distribution of colours. In order to have a canonical
representative, we choose from them the lexicographically
minimal sequence and call it the \emph{type} of $\alpha$. A
\emph{colouring type} is the type of some colour
vector~$\alpha$. The \emph{type} of a $3$-edge-colouring
$\sigma$ of $H^+$ is the type of the induced colour vector
$\sigma(f_0)\sigma(f_1)\ldots\sigma(f_5)$. A colouring type is
said to be  \emph{admissible} for $H^+$ if it is the type of a
certain $3$-edge-colouring of $H^+$. Similar definitions apply
to $C^+$ and, in fact, to all subcubic graphs with $6$
vertices of degree $1$.

\medskip\noindent
Claim 2. \emph{The colour vector induced by any
$3$-edge-colouring of $H^+$ involves all three~colours.}

\medskip\noindent
\emph{Proof of Claim 2.} Consider the colour vector induced by
a $3$-edge-colouring of $H^+$. If there was a colour that does
not occur in it, then the corresponding colour class could be
extended to a perfect matching of the entire $G$ that covers
the three uncovered edges of $C$. Consequently, $G$ could be
covered with four perfect matchings, contrary to our
assumption. Claim~2 is proved.
\medskip

\begin{figure}[htbp]
            \centering
%            \includegraphics[width=0.7\textwidth]{HPQ}
%            \caption{}
		%\includegraphics[width=0.35\textwidth]{mpole_H}\hskip1cm
		%\includegraphics[width=0.28\textwidth]{mpole_M}\hskip1cm
		\includegraphics[width=0.85\textwidth]{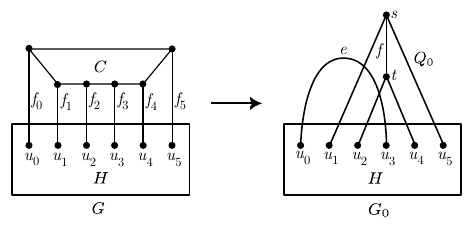}
\caption{Constructing $G_0$ from $G$.}
            \label{obr:HPQ}\end{figure}
            
\medskip\noindent
Claim 3. \emph{In the set $U=\{u_0,\ldots,u_5\}$ we have
$u_i=u_{i+3}$ or each $i\in\{0,1,2\}$.}

\medskip\noindent
\emph{Proof of Claim 3.} Let $Q$ denote the acyclic subcubic
graph on eight vertices which consists of two components:
first, the complete graph $K_2$ on two vertices joined by an
edge $e$ and second, a tree of order $6$ consisting of an edge
$f$ joining two trivalent vertices $s$ and $t$, with four
neighbouring leaves. For each $i\in\{0,1,2\}$ let us construct
an auxiliary graph $G_i$ by replacing $C_6^+$ in $G$ with a
copy $Q_i$ of $Q$ as indicated in Figure~\ref{obr:HPQ} for
$i=0$. In general, for every $i\in\{0,1,2\}$ the edges of $Q_i$
are $e=u_iu_{i+3}$, $f=st$, $su_{i+1}$, $su_{i+5}$, $tu_{i+2}$
and $tu_{i+4}$ with indices reduced modulo~$6$. Observe that
$u_i=u_{i+3}$ if and only if $G_i$ has a loop.

\begin{table}[htb]
\centering
\begin{tabular}{ccccccc}
\toprule
$Q_0$ & & $Q_1$ & & $Q_2$ & & $C_6^+$\\
\midrule
& & & & $112332$ & & $112332$\\
$122133$ & & & & & & $122133$\\
$123123$ & & $123123$ & & $123123$ & & $123123$\\
 & & $123321$ & &  & & $123321$\\
\bottomrule
\end{tabular}
\bigskip		
\caption{Colouring types for the copies of $Q$ and for $C_6^+$
involving all three colours.} \label{tbl:types}
\end{table}

Suppose to the contrary that $u_j\ne u_{j+3}$ for some
$j\in\{0,1,2\}$, that is, $G_j$ is loopless. We prove that
$G_j$ is not $3$-edge-colourable. Since $G$ is a snark, the
sets of colouring types admissible for $H^+$ and those
admissible for $C_6^+$ have no common member. Thus, by Claim~2,
for the proof that $G_j$ is not $3$-edge-colourable it is
sufficient to show that every colouring type admissible for
$Q_j$ and involving all three colours is contained in the set
of colouring types admissible for $C_6^+$. Direct verification
reveals that $Q_j$ has two such types and $C_6^+$ four and that
the required inclusion holds; see Table~\ref{tbl:types}. Hence
$G_j$ is not $3$-edge-colourable.

On the other hand, $H$ is bipartite with bipartition
$\{S,V(H)-S\}$, so $G_j-\{e,f\}$ is bipartite with bipartition
$\{S\cup\{s,t\},V(H)-S\}$. Since $G_j$ is not
$3$-edge-colourable, from Theorem~\ref{thm:abg_class1} we
deduce that $G_j$ is not almost bipartite. Thus $G_j$ has a
bridge, say~$b$. A simple counting argument shows that in $G_j$
the bridge $b$ separates $e$ from~$f$. Let $K$ be the component
of $G_j-b$ containing $e$. Now, $K-e$ is a subgraph of $G$
separated from the rest of $G$ by the $3$-edge-cut
$R_j=\{b,f_j,f_{j+3}\}=\delta_G(K-e)$. Since $K$ has at least
two vertices, and so does the other component of $G-R_j$, we
conclude that $R_j$ is a cycle-separating $3$-edge-cut in $G$.
This contradicts the assumption that $G$ is cyclically
$4$-edge-connected and establishes the fact that $u_i=u_{i+3}$
for each $i\in\{0,1,2\}$. Claim~3 is proved.

\medskip

We now finish the proof. Consider the subgraph $W$ of $G$
induced by the vertices of $C$ and $U$. Since $u_i=u_{i+3}$ for
each $i\in\{0,1,2\}$, the cut $\delta_G(W)$ comprises only
three edges. However, $G$ is cyclically $4$-edge-connected,
implying that the other component of $G-\delta_G(W)$ is just a
single vertex. Therefore
$$ |G|=|C|+|U|+1=6+3+1=10.$$
Since $G$ is a snark, it must be isomorphic to the Petersen graph.
Summing up, we have proved that every cyclically
$4$-edge-connected cubic graph $G$ with $\df{G}=3$ and
$\pi(G)\ge 5$ is the Petersen graph. The proof is complete.
\end{proof}

\begin{remark}
{\rm The proof of Theorem~\ref{thm:2} shows that in a
cyclically $4$-edge-connected cubic graph of defect $3$
different from the Petersen graph every optimal $3$-array of
perfect matchings extends to a covering with four perfect
matchings. The general problem of whether an optimal $3$-array
of perfect matchings extends to a Berge cover seems to be of
great interest.}
\end{remark}

\section{Family of snarks with defect 3 and perfect matching index 5}

\noindent{}The requirement of cyclic connectivity at least $4$
in Theorem~\ref{thm:2} is essential, because there exist
infinitely many $3$-edge-connected snarks with defect $3$ and
perfect matching index~$5$. They can be constructed as follows:

Take the Petersen graph $Pg$ and remove a vertex $v$ from it
together with the three incident edges, leaving a graph $P$
containing three $2$-valent vertices. Further, take an
arbitrary connected bipartite cubic graph on at least four
vertices, and similarly remove a vertex $u$, leaving a graph
$B$ with tree $2$-valent vertices. Create a cubic graph $H$ by
joining every $2$-valent vertex of $P$ to a $2$-valent vertex
of~$B$. Clearly, the set $R$ consisting of the newly added
edges is a cycle-separating $3$-edge-cut in $H$. We claim that
the graph $H$ has defect 3 and perfect matching index~$5$.

We first observe that $\df{H}=3$. Parity Lemma implies that $P$
is uncolourable, so $H$ is a snark, and therefore $\df{H}\ge3$.
Since $Pg$ is vertex-transitive, one can choose an optimal
array for $Pg$ in such a way that its hexagonal core
avoids the vertex~$v$. %whose removal produces the $3$-pole $R$.
As every bipartite graph is $3$-edge-colourable, one can easily
extend the $3$-array of $Pg$ to a $3$-array of $H$ which
retains the original core of $Pg$. Hence,  $\df{H}=3$.

Now we want to prove that $\pi(H)=5$. Suppose to the contrary
that $\pi(H)\le4$. Then $H$ has a covering
$\mathcal{C}=\{M_1,M_2,M_3,M_4\}$ with four perfect matchings.
Since both $P$ and $B$ contain an odd number of vertices, we
conclude that each $M_i$ has an odd number of common edges with
$R$. Moreover, since each edge of $H$ is in at most two of the
four perfect matchings, the weights of edges in the edge cut
$R$ are either $1,1,2$, or $2,2,2$. The former possibility is
excluded, because contracting $B$ to a single vertex would
produce a covering of $Pg$ with four perfect matchings,
contradicting the fact that $\pi(Pg)=5$. Therefore the weight
of each edge in $R$ equals $2$. Now, each vertex of $B$ is
incident with precisely one edge of weight $2$, including those
in $R$. It follows that the simply covered edges form a
$2$-factor of $B$, say $F$. However, $B$ has an odd number of
vertices, so at least one circuit of $F$ is odd, which is
impossible because $B$ is bipartite. Hence, $\pi(H)\ge 5$, and
from Theorem~\ref{thm:1} we infer that $\pi(H)=5$.

\section*{Acknowledgements}
This research was supported by the grant APVV-23-0076 of
Slovak Research and Development Agency. The second and the
fourth author were supported by the grant VEGA 1/0727/22.

\smallskip

The authors are grateful to Tom\'a\v s Kaiser for suggesting a
shorter proof of Proposition~\ref{prop:almost_bip} and to
anonymous referees for their constructive comments.

%\section*{References}

\end{document}